\documentclass[10pt]{article}
\usepackage[dvips]{graphicx}
\usepackage{amssymb,color}
\usepackage[latin1]{inputenc}
\usepackage{epic}
\usepackage{pstricks}
\usepackage{pst-node}
\usepackage{pstricks-add}
\usepackage[latin1]{inputenc}
\usepackage{color}
\usepackage{lettrine}
\usepackage{calc,pifont}
\usepackage[noindentafter,calcwidth]{titlesec}
\usepackage{pstricks}
\usepackage{pst-plot}
\usepackage{pst-node}
\usepackage{amsmath}
\usepackage{amssymb}
\usepackage{amsthm}
\usepackage{subfigure}
\usepackage{lineno}

\def\BBox{\kern  -0.2cm\hbox{\vrule width 0.2cm height 0.2cm}}

\newtheorem{example}{Example}
\newtheorem{remark}{Remark}

\newtheorem{teo}{Theorem}[section]
\newtheorem{coro}[teo]{Corollary}
\newtheorem{lema}[teo]{Lemma}
\newtheorem{prop}[teo]{Proposition}

\theoremstyle{definition}

\theoremstyle{remark}

\title{On ($1,C_4$) one-factorization and two orthogonal ($2,C_4$) one-factorization of complete graphs}
\author{Adrián Vázquez-Ávila\thanks{adrian.vazquez@unaq.mx}\\
{\small Subdirección de Ingeniería y Posgrado}\\
{\small Universidad Aeronáutica en Querétaro}\\
}

\date{}

\begin{document}
\maketitle
\begin{abstract}
An one-factorization $\mathcal{F}$ of the complete graph $K_n$ is ($l,C_k$), where $l\geq0$ and $k\geq4$ are integers, if the union $F\cup G$, for any $F,G\in\mathcal{F}$, includes exactly $l$ (edge-disjoint) cycles of length $k$ ($lk\leq n$). Moreover, a pair of orthogonal one-factorizations $\mathcal{F}$ and $\mathcal{G}$ of the complete graph $K_n$ is ($l,C_k$) if the union $F\cup G$, for any $F\in\mathcal{F}$ and $G\in\mathcal{G}$, includes exactly $l$ cycles of length $k$.

In this paper, we prove the following: if $q\equiv11$ (mod 24) is an odd prime power, then there is a ($1,C_4$) one-factorization of $K_{q+1}$. Also, there is a pair of orthogonal ($2,C_4$) one-factorization of $K_{q+1}$.
\end{abstract}

\textbf{Keywords.} One-factorization, orthogonal one-factorization, Strong starters.

\section{Introduction}
A \emph{one-factor} of a graph $G$ is a regular spanning subgraph of degree one. A \emph{one-factorization} of a graph $G$ is a set $\mathcal{F}=\{F_1, F_2,\ldots, F_n\}$ of edge disjoint one-factors such that $E(G)=\displaystyle\cup_{i=1}^nE(F_i)$. Two one-factorizations $\mathcal{F}=\{F_1, F_2,\ldots, F_n\}$ and $\mathcal{H}=\{H_1, H_2,\ldots, H_n\}$ of a graph $G$ are \emph{orthogonal} if $|F\cap H|\leq1$, for every $F\in\mathcal{F}$ and $H\in\mathcal{H}$. A one-factorization $\mathcal{F}=\{F_1, F_2,\ldots, F_n\}$ of a graph $G$ is ($l,C_k$), where $l\geq0$ and $k\geq4$ integers, if $F_i\cup F_j$ include exactly $l$ cycles of length $k$ ($lk\leq n$). In particular, if $l=0$, we said that the one factorization is $C_k$-free, see for example \cite{Bao,Meszka,AvilaC4free}.

A pair of orthogonal one-factorizations $\mathcal{F}$ and $\mathcal{H}$ of the complete graph $K_n$ is ($l,C_k$), where $l\geq0$ and $k\geq4$ integers, if $F\cup H$ include exactly $l$ cycles of length $k$, for all $F\in\mathcal{F}$ and $H\in\mathcal{H}$.

An interesting way of constructing one-factorizations of the complete graph is by using starters of additive Abelian groups of odd order: Let $\Gamma$ be a finite additive Abelian group of odd order $n=2k+1$. And let $\Gamma^*=\Gamma\setminus\{0\}$ be the set of non-zero elements of $\Gamma$. A \emph{starter} for $\Gamma$ is a set $S=\{\{x_1,y_1\},\ldots,\{x_k,y_k\}\}$ such that
$\left\{x_1,\ldots,x_k,y_1,\ldots,y_k\right\}=\Gamma^*$
and $\left\{\pm(x_i-y_i):i=1,\ldots,k\right\}=\Gamma^*$. Moreover, if $\left\{x_i+y_i:i=1,\ldots,k\right\}\subseteq\Gamma^*$ with $\left|\left\{x_i+y_i:i=1,\ldots,k\right\}\right|=k$, then $S$ is called \emph{strong starter} for $\Gamma$.  To see some works related to strong starters, the reader may consult \cite{Avila,MR0325419,dinitz1984,MR1010576,MR0392622,MR1044227,MR808085,MR0249314,MR0260604,Skolem,AvilaSkolem,AvilaSkolem2,AvilaSkolem3}.	

Strong starters were first introduced by Mullin and Stanton in \cite{MR0234587} in constructing of Room
squares. Starters and strong starters have been useful to built many combinatorial designs such as Room cubes \cite{MR633117}, Howell designs \cite{MR728501}, Kirkman triple systems \cite{MR808085,MR0314644}, Kirkman squares and cubes \cite{MR833796,MR793636},  and factorizations of complete graphs \cite{MR0364013,Bao,MR2206402,MR1010576,MR685627,AvilaC4free}.

Let $\Gamma$ be a finite additive Abelian group of odd order $n=2k+1$. It is well known that $F_\gamma=\{\{\infty, \gamma\}\}\cup\{\{x_i+\gamma,y_i+\gamma\} : 1 \leq i\leq k\}$, for all $\gamma\in\Gamma$, forms a one-factorization of the complete graph on $\Gamma\cup\{\infty\}$. Hence, if $F_0=\{\{\infty,0\}\}\cup\{\{x_i,y_i\} : 1 \leq i\leq k\}$, then $F_\gamma=F_0+\gamma$, for all $\gamma\in\Gamma^*$. On the other hand, let $S=\{\{x_i,y_i\}: 1\leq i\leq k\}$ and $T=\{\{u_i,v_i\}: 1\leq i\leq k\}$ two starters for $\Gamma$. Without loss of generality, we assume that $x_i-y_i=u_i-v_i$, for all $i=1,\ldots,k$. Then $S$ and $T$ are \emph{orthogonal starters} if $u_i-x_i=u_j-x_j$ implies $i=j$, and if $u_i\neq x_i$, for all $i=1,\ldots,k$.

Let $S=\{\{x_i,y_i\}:i=1,\ldots,k\}$ be a starter for a finite additive Abelian group $\Gamma$ of odd order $n=2k+1$. It is not hard to see that $-S=\{\{-x_i,-y_i\}:i=1,\ldots,k\}$ is also a starter for $\Gamma$. 

\begin{teo}\cite{MR623318}
If there is a strong starter $S$ in an additive Abelian group of odd order, then $S$ and $-S$ are pairwise orthogonal starters.		
\end{teo}

Let $q$ be an odd prime power. An element $x\in\mathbb{F}_q^*$ is called a \emph{quadratic residue} if there exists an element $y\in\mathbb{F}_q^{*}$ such that $y^2=x$. If there is no such $y$, then $x$ is called a \emph{non-quadratic residue.} The set of quadratic residues of $\mathbb{F}_q^{*}$ is denoted by $QR(q)$ and the set of non-quadratic residues is denoted by $NQR(q)$. It is well known that $QR(q)$ is a cyclic subgroup of $\mathbb{F}_q^{*}$ of order $\frac{q-1}{2}$ (see for example \cite{MR2445243}). Also, it is well known that if either $x,y\in QR(q)$ or $x,y\in NQR(q)$, then $xy\in QR(q)$. Also, if $x\in QR(q)$ and $y\in NQR(q)$, then $xy\in NQR(q)$. For more details of this kind of results the reader may consult \cite{burton2007elementary,MR2445243}.

Horton in \cite{MR623318} proved the following (see too \cite{Avila}):

\begin{prop}\cite{MR623318}\label{prop:Horton}
If $q\equiv3$ (mod 4) is an odd prime power ($q\neq3$) and $\beta\in NQR(q)\setminus\{-1\}$, then 
	\begin{eqnarray*}\label{strong_1}
		S_\beta=\left\{\{x,x\beta\}:x\in QR(q)\right\},
	\end{eqnarray*} 
is a strong starter for $\mathbb{F}_q$.
\end{prop}

\begin{teo}\cite{MR623318}
Let $q\equiv3$ (mod 4) be a prime power ($q\neq3$). If $\beta_1,\beta_2\in NQR(q)\setminus\{-1\}$, with $\beta_1\neq\beta_2$, then $S_{\beta_1}$ and $S_{\beta_2}$ are orthogonal. 
\end{teo}

In this paper, we prove the following:

\begin{teo}
If $q\equiv$3 (mod 8) is an odd prime power such that $q\equiv-1$ (mod 12), then there is $(1,C_4)$ one-factorization of $K_{q+1}$.   
\end{teo} 

\begin{teo}
If $q\equiv$3 (mod 8) is an odd prime power such that $q\equiv-1$ (mod 12), then there is a pair of orthogonal $(2,C_4)$ one-factorization of $K_{q+1}$.
\end{teo}

\section{$(1,C_4)$ one-factorizations of complete graph}\label{sec:main}
Let $n\geq11$ be an odd number. A one-factorization $\mathcal{F}=\{F_1,F_2\ldots,F_n\}$ of $K_{n+1}$ is \emph{uniform} (also called \emph{semi-regular}), if for any $F_i,F_j,F_k,F_m\in\mathcal{F}$ such that $i\neq j$ and $k\neq m$ satisfy that $F_i\cup F_j$ and $F_k\cup F_m$ generate the same cycle structure. Rosa in \cite{Rosa} shows several examples of infinite families of uniform one-factorizations of the complete graph $K_{2n}$. An important example of uniform one-factorization of complete graph $K_{2n}$ is the one-factorization generated by the Mullin-Nemeth strong starters \cite{MR0249314}: Let $p$ be an odd prime and let $m$ be a positive integer such that $3<p^m=2t+1$, where $t$ is odd such that $p^m\equiv3$ (mod 4). If $r$ is a primitive root in $\mathbb{F}_{p^m}$, then $$S=\{\{r^0,r\},\{r^2,r^3\},\ldots,\{r^{2t-2},r^{2t-1}\}\}$$ is a strong starter for $\mathbb{F}_{p^m}$, see for instance \cite{Avila,MR0249314}.

The following proof is completely analogous to the proof of Theorem 2.1 given Anderson in \cite{MR0364013}.

\begin{teo}\label{theorem:uniforme}
Let $q\equiv3$ (mod 4) be an odd prime power $(q\neq3)$ and let $y\in QR(q)$ be a generator of $QR(q)$. If $\beta\in NQR(q)\setminus\{-1\}$ then the cycle structure of $F_0\cup F_0+y$ is the same as the cycle structure of $(F_0+g)\cup(F_0+h)$, for every different elements $g,h\in\mathbb{F}_q$. That is, the one-factorization generated by the starter $S_\beta$ is uniform.
\end{teo}

\begin{lema}\label{lemma:uno}
Let $q\equiv$3 (mod 4) be an odd prime power with $q\geq11$ and $\beta\in NQR(q)\setminus\{-1\}$. If $\beta^2-\beta+1=0$, then the one-factorization generated by the starter $S_\beta$ include a cycle of length four, $C_4$, with $\infty\in V(C_4)$. 
\end{lema}

\begin{proof}
	Let $F_i=\{\{\infty,i\}\}\cup\{\{x+i,x\beta+i\}: x\in QR(q)\}$, for $i\in\mathbb{F}_q$. Then
	$\mathcal{F}=\{F_i:i\in\mathbb{F}_q\}$ is a one-factorization of the complete graph on $\mathbb{F}_q\cup \{\infty\}$. By Theorem \ref{theorem:uniforme} it is sufficient to prove that $F_0\cup F_i$, where $i\in QR(q)$ is a generator of $QR(q)$, include a cycle of length four, $C_4$, such that $\infty\in V(C_4)$. 
	
	Hence, $\{\infty,0\},\{\infty,i\}\in E(C_4)$, where $\{\infty,0\}\in F_0$ and $\{\infty,i\}\in F_i$. Then $\{-i+i,-i\beta^{-1}+i\},\{i,i\beta\}\in E(C_4)$. Therefore, $F_0\cup F_i$ include the following cycle of length 4:  $$\{\infty,0\}\{0,-i\beta^{-1}+i\}\{i,i\beta\}\{\infty,i\}=C_4,$$which implies that $i\beta^{-1}+i=i\beta$. Hence, $\beta^2-\beta+1=0$.
\end{proof}

\begin{lema}\label{lemma:3QR}
	If $q\equiv-1$ (mod 12) is an odd prime power with $q\geq11$ and $\beta\in NQR(q)\setminus\{-1\}$, then $\beta^2-\beta+1\neq0$.	
\end{lema}
\begin{proof}
Suppose that there exists $\beta\in NQR(q)\setminus\{-1\}$ such that $\beta^2-\beta+1=0$. Notice that$$\beta(\beta-2)=\beta^2-2\beta=-(\beta+1).$$ Given that $\beta\in NQR(q)$ and $-1\in NQR(q)$, then $(\beta-2)(\beta+1)\in QR(q)$. Furthermore, $(\beta-2)(\beta+1)=-3$, which implies that $(\beta-2)(\beta+1)\in NQR(q)$ (since $q\equiv-1$ mod 12 implies that $3\in QR(q)$), which is a contradiction.		
\end{proof}

It is not hard to check that, if $q\equiv1$ (mod 3) is an odd prime power with $q\geq7$, then there exist $\beta\in NQR(q)\setminus\{-1\}$ such that $\beta^2-\beta+1=0$. Hence, we have the following:

\begin{teo}\label{theorem:noinf}
Let $q\equiv$3 (mod 8) be an odd prime power with $q\geq11$ such that $3\in NQR(q)$ (such that $q\equiv1$ (mod 3)). If $\beta\in NQR(q)\setminus\{-1\}$ is such that $\beta^2-\beta+1=0$, then the one-factorization generated by the starter $S_\beta$ is $(1,C_4)$. 
\end{teo}

\begin{proof}
	Let $F_i=\{\{\infty,i\}\}\cup\{\{x+i,x\beta+i\}: x\in QR(q)\}$, for $i\in\mathbb{F}_q$. Then
	$\mathcal{F}=\{F_i:i\in\mathbb{F}_q\}$ is a one-factorization of the complete graph on $\mathbb{F}_q\cup \{\infty\}$. By Theorem \ref{theorem:uniforme} it is sufficient to prove that $F_0\cup F_i$, where $i\in QR(q)$ is a generator of $QR(q)$, include an unique cycle of length four, $C_4$. 
	
	Suppose that there is a $\beta\in NQR(q)\setminus\{-1\}$ such that $\beta^2-\beta+1=0$ and $S_\beta$ include a cycle of length four, $C_4$, with $\infty\neq V(C_4)$ (by Lemma \ref{lemma:beta2}). Given that $E(C_4)\cap F_0\neq\emptyset$, then there is an $a\in V(C_4)\cap QR(q)$ such that
	\begin{center}
		$\{a,a\beta\}\in F_0$ and $\{a\beta,(a\beta-i)\beta^{t_1}+i\}\in F_i$,
	\end{center}
	with $t_1=1$ if $a\beta-i\in QR(q)$, and $t_1=-1$ if $a\beta-i\in NQR(q)$. On the other hand
	\begin{center}
		$\{a,(a-i)\beta^{t_2}+i\}\in F_i$ and $\{(a-i)\beta^{t_2}+i,((a-i)\beta^{t_2}+i)\beta^{t_3}\}\in F_0$,	
	\end{center}
	whit $t_2=1$ if $a-i\in QR(q)$, and $t_2=-1$ if $a-i\in NQR(q)$, and $t_3=1$ if $(a-i)\beta^{t_2}+i\in QR(q)$, and $t_3=-1$ if $(a-i)\beta^{t_2}+i\in NQR(q)$. Therefore
	
	\begin{equation*}
	a(\beta^{t_1+1}-\beta^{t_2+t_3})+i(\beta^{t_2+t_3}-\beta^{t_3}-\beta^{t_1}+1)=0,
	\end{equation*}
	where $t_1,t_2,t_3\in\{-1,1\}$. Notice that, if $\{a,b\}\{b,c\}\{a,d\}\{d,c\}=C_4$, then $|V(C_4)\cap QR(q)|=2$. It is not hard to check that, if $t_1+1=t_2+t_3$, where $t_1,t_2,t_3\in\{-1,1\}$, then $\beta^{t_2+t_3}-\beta^{t_3}-\beta^{t_1}+1\neq0$, which implies that the one-factorization generated by the starter $S_\beta$ is $C_4$-free. Suppose that $t_1+1\neq t_2+t_3$.
	
	\begin{itemize}
		\item [case (i)] Suppose that $t_1=t_2=1$ and $t_3=-1$. In this case $(a-i)\beta+i\in QR(q)$ and $(a\beta-i)\beta+i\in NQR(q)$. Since $\{(a-i)\beta+i,(a\beta-i)\beta+i\}\in F_0$, then $(a\beta-i)\beta+i-((a-i)\beta+i)=x(\beta-1)$, where $x=-a\beta$. Hence $(a-i)\beta+i=-a\beta$ and $(a\beta-1)\beta+i=-a\beta^2$, since $\{x,x\beta\}\in F_0$, where $x\in QR(q)$. But $(a\beta-1)\beta+i=-a\beta^2$ implies that $2a\beta^2=i(\beta-1)$, which is a contradiction, since $2a\beta^2\in NQR(q)$ (given that $2\in NQR(q)$) and $i(\beta-1)\in QR(q)$ (given that $\beta^2-\beta+1=0$). Hence, the one-factorization generated by the starter $S_\beta$ is ($0,C_4$), $C_4$-free.
		
		\
		
		\item [case (ii)] Suppose that $t_1=t_3=1$ and $t_2=-1$. In this case $(a-i)\beta^{-1}+i\in QR(q)$ and $(a\beta-i)\beta+i\in NQR(q)$. Since 
		$\{x,x\beta\}\in F_0$, where $x\in QR(q)$, then $x=(a-i)\beta^{-1}+i$ and $x\beta=(a\beta-i)\beta+i$. Hence, $(a-i)+i\beta=(a\beta-i)\beta+i$, which implies that $2i\beta=a(\beta^2+\beta)=a(2\beta-1)$ (since $\beta^2-\beta+1=0$), which implies that $2\beta(a-i)=a$, a contradiction (since $a-i\in NQR(q)$). Hence, the one-factorization generated by the starter $S_\beta$ is $C_4$-free.
		
		\
		
		\item [case (iii)] Suppose that $t_2=t_3=1$ and $t_1=-1$. In this case $(a-i)\beta+i\in QR(q)$ and $(a\beta-i)\beta^{-1}+i\in NQR(q)$. Since 
		$\{x,x\beta\}\in F_0$, where $x\in QR(q)$, then $x=(a-i)\beta+i$ and $x\beta=(a\beta-i)\beta^{-1}+i$. Hence, $(a-i)\beta^2+i\beta=(a\beta-i)\beta^{-1}+i$, which implies that $a(\beta-1)=a\beta-i$, a contradiction (since $a\beta-i\in NQR(q)$). Hence, the one-factorization generated by the starter $S_\beta$ is $C_4$-free.
		
		\
		
		\item [case (iv)] Suppose that $t_1=1$ and $t_2=t_3=-1$. Then $$a(\beta^2-\beta^{-2})+i(\beta^{-2}-\beta^{-1}-\beta+1)=0$$ Given that $\beta^2-\beta+1=0$ (implying that $\beta+\beta^{-1}=1$), then $a(\beta^2-\beta^{-2})+i\beta^{-2}=0$, which implies that $a\beta^2=\beta^{-2}(a-i)$, a contradiction (since $a-i\in NQR(q)$). Hence, the one-factorization generated by the starter $S_\beta$ is $C_4$-free.
		
		\
		
		\item [case (v)]  Suppose that $t_1=t_2=t_3=-1$. Then $$a(1-\beta^{-2})+i(\beta^{-2}-2\beta^{-1}+1)=0,$$which implies that
		$a(\beta^2-1)+i(\beta^2-2\beta+1)=a(\beta-2)-i\beta=0$ (since $\beta^2-\beta+1=0$). Hence, $\beta(a-i)=2a$, which is a contradiction (since $a-i\in NQR(q)$ and $2\in NQR(q)$). Hence, the one-factorization generated by the starter $S_\beta$ is $C_4$-free.
	\end{itemize}
	Hence, by Lemma \ref{lemma:uno} the one-factorization generated by the starter $S_\beta$ is $(1,C_4)$.
\end{proof}	

\begin{example}
Let $q=19$. A primitive root of $\mathbb{F}_{19}$ is $r=2$ and $i=r^2=4$ is a generator of $QR(19)$. Furthermore, the set of non-quadratic residues of $\mathbb{F}_{19}$ is: $$\{2,3,8,10,12,13,14,15,18\}.$$ Hence, if $\beta=8$, then $\beta^2-\beta+1=0$, which implies that the one-factorization generated by the starter $S_8$ include an unique cycle $C_4$, with $V(C_4)=\{\infty,0,4,13\}$. 
\end{example}

\begin{lema}\label{lemma:final}
Let $q\equiv$3 (mod 8) be an odd prime power such that $q\equiv-1$ (mod 12). If there exist $\beta\in NQR(q)\setminus\{-1\}$ such that $\beta-1\in NQR(q)$ and $\beta^2+1\in QR(q)$, then the one-factorization generated by the starter $S_\beta$ is $(1,C_4)$.  
\end{lema}

\begin{proof}
	Let $F_i=\{\{\infty,i\}\}\cup\{\{x+i,x\beta+i\}: x\in QR(q)\}$, for $i\in\mathbb{F}_q$. Then
	$\mathcal{F}=\{F_i:i\in\mathbb{F}_q\}$ is a one-factorization of the complete graph on $\mathbb{F}_q\cup \{\infty\}$. Then the one-factorization generated by the starter $S_\beta$ include a cycle of length four, $C_4$, with $\infty\not\in V(C_4)$ (by Lemma \ref{lemma:3QR} and Theorem \ref{theorem:noinf}). 
	
	By Theorem \ref{theorem:uniforme}, if $i\in QR(q)$ is a generator of $QR(q)$, then the cycle structure $F_0\cup F_i$ is the same as the cycle structure of the union of any two distinct one-factors, say $F_g\cup F_h$. Given that $E(C_4)\cap F_0\neq\emptyset$, then there is an $a\in V(C_4)\cap QR(q)$ such that
	\begin{center}
		$\{a,a\beta\}\in F_0$ and $\{a\beta,(a\beta-i)\beta^{t_1}+i\}\in F_i$,
	\end{center}
	with $t_1=1$ if $a\beta-i\in QR(q)$, and $t_1=-1$ if $a\beta-i\in NQR(q)$. On the other hand
	\begin{center}
		$\{a,(a-i)\beta^{t_2}+i\}\in F_i$ and $\{(a-i)\beta^{t_2}+i,((a-i)\beta^{t_2}+i)\beta^{t_3}\}\in F_0$,	
	\end{center}
	whit $t_2=1$ if $a-i\in QR(q)$, and $t_2=-1$ if $a-i\in NQR(q)$, and $t_3=1$ if $(a-i)\beta^{t_2}+i\in QR(q)$, and $t_3=-1$ if $(a-i)\beta^{t_2}+i\in NQR(q)$. Therefore
	
	\begin{equation*}
	a(\beta^{t_1+1}-\beta^{t_2+t_3})+i(\beta^{t_2+t_3}-\beta^{t_3}-\beta^{t_1}+1)=0,
	\end{equation*}
	where $t_1,t_2,t_3\in\{-1,1\}$. Notice that, if $\{a,b\}\{b,c\}\{a,d\}\{d,c\}=C_4$, then $|V(C_4)\cap QR(q)|=2$. It is not hard to check that, if $t_1+1=t_2+t_3$, where $t_1,t_2,t_3\in\{-1,1\}$, then $\beta^{t_2+t_3}-\beta^{t_3}-\beta^{t_1}+1\neq0$, which implies that the one-factorization generated by the starter $S_\beta$ is $C_4$-free. Suppose that $t_1+1\neq t_2+t_3$.
	\begin{itemize}
		\item[Case (i): ] Suppose that $t_1=1$ and $-t_2=t_3=1$. In this case $(a-i)\beta^{-1}+i\in QR(q)$ and $(a\beta-i)\beta+i\in NQR(q)$. Since 
		$\{x,x\beta\}\in F_0$, where $x\in QR(q)$, then $x=(a-i)\beta^{-1}+i$ and $x\beta=(a\beta-i)\beta+i$. Hence, $(a\beta-i)\beta+i=(a-i)+i\beta$, which implies that $2i=a(\beta+1)$. Therefore, if $\beta+1\in NQR(q)$, then the one-factorization generated by the starter $S_\beta$ does include a cycle of length four; otherwise don't.
		
		\
		
		\item [Case (ii): ] Suppose that $t_1=1$ and $t_2=t_3=-1$. In this case $(a-i)\beta^{-1}+i\in NQR(q)$ and $(a\beta-i)\beta+i\in QR(q)$. Since 
		$\{x,x\beta\}\in F_0$, where $x\in QR(q)$, then $x=(a\beta-i)\beta+i$ and $x\beta=(a-i)\beta^{-1}+i$. Hence, $(a-i)\beta^{-1}+i=(a\beta-i)\beta^2+i\beta$, which implies that $$i(\beta^3-\beta^2+\beta-1)=a(\beta^4-1)=a(\beta+1)(\beta^3-\beta^2+\beta-1)$$ implying that $a-i=-a\beta$, which is a contradiction, given that $t_2=-1$ implies that $a-i\in NQR(q)$. the one-factorization generated by the starter $S_\beta$ is $C_4$-free.
		
		\
		
		\item [Case (iii): ] Suppose that $t_1=1$ and $t_2=-t_3=1$. In this case $(a-i)\beta+i\in NQR(q)$ and $(a\beta-i)\beta+i\in QR(q)$. Since 
		$\{x,x\beta\}\in F_0$, where $x\in QR(q)$, then $x=(a\beta-i)\beta+i$ and $x\beta=(a-i)\beta+i$. Hence, $(a-i)\beta+i=(a\beta-i)\beta^2+i\beta$, which implies that $i(\beta-1)=a\beta(\beta+1)$. If $\beta+1\in QR(q)$, then the one-factorization generated by the starter $S_\beta$ does include a cycle of length four; otherwise don't.
		
		\
		
		\item [Case (iv): ] Suppose that $t_1=-1$ and $t_2=t_3=1$. In this case $(a-i)\beta+i\in QR(q)$ and $(a\beta-i)\beta^{-1}+i\in NQR(q)$. Since 
		$\{x,x\beta\}\in F_0$, where $x\in QR(q)$, then $x=(a-i)\beta+i$ and $x\beta=(a\beta-i)\beta^{-1}+i$. Hence, $(a\beta-i)\beta^{-1}+i=(a-i)\beta^2+i\beta$, which implies that $a\beta(\beta+1)=i(\beta^2+1)$. If $\beta+1\in NQR(q)$, then the one-factorization generated by the starter $S_\beta$ does include a cycle of length four; otherwise don't.
		
		\
		
		\item [Case (v): ] Suppose that $t_1=t_2=t_3=-1$. In this case $(a-i)\beta^{-1}+i\in NQR(q)$ and $(a\beta-i)\beta^{-1}+i\in QR(q)$. Since 
		$\{x,x\beta\}\in F_0$, where $x\in QR(q)$, then $x=(a\beta-i)\beta^{-1}+i$ and $x\beta=(a-i)\beta^{-1}+i$. Hence, $(a-i)\beta^{-1}+i=(a\beta-i)+i\beta$, which implies that $i(\beta-1)^2=0$, which is a contradiction. 
	\end{itemize}
	
	It is not hard to check that Case (i) and Case (iv) generate the same (unique) cycle of length four. On the other hand, Case (iii) generate an unique cycle of length four. Therefore, the pair of orthogonal one-factorizations generated by the starter $S_\beta$ and $-S_\beta$ does include exactly one cycle of length four.
\end{proof}	

\begin{example}
	Let $q=59$. A primitive root of $\mathbb{F}_{59}$ is $r=2$ and $i=r^2=4$ is a generator of $QR(59)$. Hence, if $\beta=32$, then $\beta-1\in NQR(59)$ and $\beta^2+1\in QR(q)$. Given that $a=36\in QR(59)$ is such that $2i=a(\beta-1)$, then the one-factorization generated by the starter $S_{32}$ include an unique cycle $C_4$, with $V(C_4)=\{5,31,36,42\}$. 
\end{example}

\begin{remark}\label{remark}
	Let $q=ef+1$ be a prime power and let $H$ be the subgroup of $\mathbb{F}_q^*$ of order $f$ with $\{H=C_0,\ldots,C_{e-1}\}$ the set of (multiplicative) cosets of $H$ in $\mathbb{F}_q^*$ (that is, $C_i = g^iC_0$, where $g$ is the least primitive element of $\mathbb{F}_q$). The \emph{cyclotomic number}
	$(i,j)$ is defined as $|\{x\in C_i: x+1\in C_j\}|$. In particular, if $e=2$ and $f$ is odd, then $(0,0)=\frac{f-1}{2}$, $(0,1)=\frac{f+1}{2}$, $(1,0)=\frac{f-1}{2}$ and $(1,1)=\frac{f-1}{2}$. On the other hand, if $e=2$ and $f$ is even, then $(0,0)=\frac{f-2}{2}$, $(0,1)=\frac{f}{2}$, $(1,0)=\frac{f}{2}$ and $(1,1)=\frac{f}{2}$, see \cite{book:206537}, Table VII.8.50.
\end{remark}

Hence, if $e=2$, then $C_0=QR(q)$ and $C_1=NQR(q)$ are the cosets of $QR(q)$ in $\mathbb{F}_q^*$.

\begin{lema}\label{lemma:conjunto_no_vacio}
	Let $q\equiv$3 (mod 8) be an odd prime power such that $q\equiv-1$ (mod 12). Then the set $$M=\{\beta\in NQR(q)\setminus\{-1\}: \beta-1\in NQR(q),\beta^2+1\in QR(q)\}\neq\emptyset$$
\end{lema}

\begin{proof}
	Given that $-1\in NQR(q)$, then $\beta\to\beta^2$, for all $\in\beta\in NQR(q)$, is a bijection between $NQR(q)$ and $QR(q)$. Let $A=\{\alpha\in QR(q):\alpha^2+1,\alpha^2-1\in QR(q)\}$ (see Remark \ref{remark}). Assume that $\alpha-1\in QR(q)$ (which implies that $\alpha+1\in QR(q)$), otherwise, if $\alpha-1\in NQR(q)$, then $M\neq\emptyset$. Moreover, assume that $\alpha^4+1\in NQR(q)$, otherwise $\beta=-\alpha^2\in M$. Given that $\alpha-1,\alpha+1\in QR(q)$, then $\alpha^2-1\in QR(q)$, which implies that $(\alpha^2-1)(\alpha^2+1)\in QR(q)$ and $\alpha^4+1\in NQR(q)$, a contradiction.
\end{proof}

Hence, by Theorem \ref{theorem:noinf}, Lemma \ref{lemma:final} and Lemma \ref{lemma:conjunto_no_vacio}, we have:

\begin{teo}\label{theorem:main_1}
Let $q\equiv$3 (mod 8) be an odd prime power such that $q\equiv-1$ (mod 12). Then there exist a one-factorization $(1,C_4)$ of $K_{q+1}$.
\end{teo}

\section{$(2,C_4)$ two orthogonal one-factorizations of complete graph}
In this section we will prove that the strong starter given by Horton (see Proposition \ref{prop:Horton}) doesn't generate two orthogonal $(1,C_4)$ one-factorization of the complete graph; however, does generate two orthogonal $C_4$-free (see for instance \cite{Bao,AvilaC4free}) and $(2,C_4)$ one-factorizations of complete graphs.

\begin{lema}\label{lemma:beta,betamedios,orto}
	Let $q\equiv3$ (mod 8) be an odd prime power and $\beta\in NQR(q)\setminus\{-1\}$. If $\beta\not\in\{2,2^{-1}\}$, then the pair of orthogonal one-factorizations generated by the starter $S_\beta$ and $-S_\beta$ doesn't include a cycle of length four, $C_4$, with $\infty\not\in V(C_4)$.
\end{lema}

\begin{proof}
	Let $F_i=\{\{\infty,i\}\}\cup\{\{x+i,x\beta+i\}: x\in QR(q)\}$ and $G_i=\{\{\infty,i\}\}\cup\{\{y+i,y\beta+i\}: y\in NQR(q)\}$, for $i\in\mathbb{F}_q$. Then $\mathcal{F}=\{F_i:i\in\mathbb{F}_q\}$ and $\mathcal{G}=\{G_i:i\in\mathbb{F}_q\}$ are orthogonal one-factorizations of the complete graph on $\mathbb{F}_q \cup\{\infty\}$. Given that $F_i=F_0+i$ and $G_i=G_0+i$, without loss of generality we can suppose that $F_0\cup G_i$, where $i\in\mathbb{F}(q)$, doesn't include a cycle of length 4, $C_4$, such that $\infty\not\in V(C_4)$. 
	
	Assume to the contrary that $F_0\cup G_i$, for some $i\in\mathbb{F}_q$, does include a cycle of length 4, $C_4$, such that $\infty\in V(C_4)$. Consequently $\{\infty,0\},\{\infty,i\}\in E(C_4)$, where $\{\infty,0\}\in F_0$ and $\{\infty,i\}\in G_i$. Then $\{0,-i\beta^{t_1}+i\},\{i,i\beta^{t_2}\}\in E(C_4)$, where $t_1=1$ if $-i\in NQR(q)$ and $t_1=-1$ if $-i\in QR(q)$, and $t_2=1$ if $i\in QR(q)$ and $t_2=-1$ if $i\in NQR(q)$. Given that, if $i\in QR(q)$ then $t_2=1$ and $t_1=1$, and if $i\in NQR(q)$ then $t_2=-1$ and $t_1=-1$. Therefore $-i\beta^{t_1}+i=i\beta^{t_2}$, where $t_1,t_2\in\{-1,1\}$ with $t_1=t_2=t$. Hence, $2\beta^t=1$ with $t\in\{-1,1\}$, which is a contradiction to the hypothesis.
\end{proof}

\begin{lema}\label{lemma:final_2}
	Let $q\equiv3$ (mod 8) be an odd prime power with $q\equiv-1$ (mod 12) and $\beta\in NQR(q)\setminus\{-1\}$. If $\beta\not\in\{2^{-1},2\}$ is such that $\beta+1,\beta-1\in NQR(q)$ and $\beta^2+1\in QR(q)$, then the pair of orthogonal one-factorizations generated by the starter $S_\beta$ and $-S_\beta$ does include exactly two cycles of length four.
\end{lema}
\begin{proof}
	Let $F_i=\{\{\infty,i\}\}\cup\{\{x+i,x\beta+i\}: x\in QR(q)\}$ and $G_i=\{\{\infty,i\}\}\cup\{\{y+i,y\beta+i\}: y\in NQR(q)\}$, for $i\in\mathbb{F}_q$. Then $\mathcal{F}=\{F_i:i\in\mathbb{F}_q\}$ and $\mathcal{G}=\{G_i:i\in\mathbb{F}_q\}$ are orthogonal one-factorizations of the complete graph on $\mathbb{F}_q \cup\{\infty\}$. Given that $F_i=F_0+i$ and $G_i=G_0+i$, without loss of generality we can suppose that $F_0\cup G_i$, where $i\in\mathbb{F}(q)$, does include one cycle of length four, $C_4$, such that $\infty\not\in V(C_4)$ (see Lemma \ref{lemma:beta,betamedios,orto}). Given that $E(C_4)\cap F_0\neq\emptyset$, then there is an $a\in V(C_4)\cap QR(q)$ such that
	\begin{center}
		$\{a,a\beta\}\in F_0$ and $\{a\beta,(a\beta-i)\beta^{t_1}+i\}\in G_i$,
	\end{center}
	whit $t_1=1$ if $a\beta-i\in NQR(q)$, and $t_1=-1$ if $a\beta-i\in QR(q)$. On the other hand,
	\begin{center}
		$\{a,(a-i)\beta^{t_2}+i\}\in G_i$ and $\{(a-i)\beta^{t_2}+i,((a-i)\beta^{t_2}+i)\beta^{t_3}\}\in F_0$,	
	\end{center}
	with $t_2=1$ if $a-i\in NQR(q)$ and $t_2=-1$ if $a-i\in QR(q)$, and $t_3=1$ if $(a-i)\beta^{t_2}+i\in QR(q)$ and $t_3=-1$ if $(a-i)\beta^{t_2}+i\in NQR(q)$. Therefore
	\begin{equation*}
	a(\beta^{t_1+1}-\beta^{t_2+t_3})+i(\beta^{t_2+t_3}-\beta^{t_3}-\beta^{t_1}+1)=0,
	\end{equation*}
	for all $t_1,t_2,t_3\in\{-1,1\}$. Notice that, if $\{a,b\}\{b,c\}\{a,d\}\{d,c\}=C_4$, then $|V(C_4)\cap QR(q)|=2$. It is not hard to check that, if $t_1+1=t_2+t_3$, where $t_1,t_2,t_3\in\{-1,1\}$, then $\beta^{t_2+t_3}-\beta^{t_3}-\beta^{t_1}+1\neq0$, which implies that the pair of orthogonal one-factorizations generated by the starter $S_\beta$ and $-S_\beta$ doesn't include a cycle of length four.  Suppose that $t_1+1\neq t_2+t_3$.
	
	\begin{itemize}
		\item [Case (i):] Suppose that $t_1=t_2=1$ and $t_3=-1$. In this case $(a\beta-i)\beta+i\in QR(q)$ and $(a-i)\beta+i\in NQR(q)$. Since 
		$\{x,x\beta\}\in F_0$, where $x\in QR(q)$, then $x=(a\beta-i)\beta+i$ and $x\beta=(a-i)\beta+i$. Hence, $(a\beta-i)\beta^2+i\beta=(a-i)\beta+i$, which implies that $a\beta(\beta+1)=i(\beta-1)$. If $i\in NQR(q)$, then the pair of orthogonal one-factorizations generated by the starter $S_\beta$ and $-S_\beta$ does include a cycle of length four; otherwise don't.
		
		\ 
		
		\item [Case (ii):] Suppose that $t_1=t_3=1$ and $t_2=-1$. In this case $(a\beta-i)\beta+i\in NQR(q)$ and $(a-i)\beta^{-1}+i\in QR(q)$. Since 
		$\{x,x\beta\}\in F_0$, where $x\in QR(q)$, then $x=(a-i)\beta^{-1}+i$ and $x\beta=(a\beta-i)\beta+i$. Hence, $(a\beta-i)\beta+i=(a-i)+i\beta$, which implies that $2i=a(\beta+1)$. If $i\in QR(q)$, then the pair of orthogonal one-factorizations generated by the starter $S_\beta$ and $-S_\beta$ does include a cycle of length four; otherwise don't.
		
		\
		
		\item [Case (iii):] Suppose that $t_1=1$ and $t_2=t_3=-1$. In this case $(a\beta-i)\beta+i\in QR(q)$ and $(a-i)\beta^{-1}+i\in NQR(q)$. Since 
		$\{x,x\beta\}\in F_0$, where $x\in QR(q)$, then $x=(a\beta-i)\beta+i$ and $x\beta=(a-i)\beta^{-1}+i$. Hence, $(a\beta-i)\beta^2+i\beta=(a-i)\beta^{-1}+i$, which implies that $i=a(\beta+1)$. If $i\in NQR(q)$, then the pair of orthogonal one-factorizations generated by the starter $S_\beta$ and $-S_\beta$ does include a cycle of length four; otherwise don't.

		\
		
		\item [Case (iv):] Suppose that $t_2=t_3=1$ and $t_1=-1$. In this case $(a-i)\beta+i\in QR(q)$ and $(a\beta-i)\beta^{-1}+i\in NQR(q)$. Since 
		$\{x,x\beta\}\in F_0$, where $x\in QR(q)$, then $x=(a-i)\beta+i$ and $x\beta=(a\beta-i)\beta^{-1}+i$. Hence, $(a-i)\beta^2+i\beta=(a\beta-i)\beta^{-1}+i$, which implies that $a\beta(\beta+1)=i(\beta^2+1)$. If $i\in QR(q)$, then the pair of orthogonal one-factorizations generated by the starter $S_\beta$ and $-S_\beta$ does include a cycle of length four; otherwise don't.
		
		\
		
		\item [Case (v): ] Suppose that $t_1=t_2=t_3=-1$. In this case $(a\beta-i)\beta^{-1}+i\in QR(q)$ and $(a-i)\beta^{-1}+i\in NQR(q)$. Since 
		$\{x,x\beta\}\in F_0$, where $x\in QR(q)$, then $x=(a\beta-i)\beta^{-1}+i$ and $x\beta=(a-i)\beta^{-1}+i$. Hence, $(a\beta-i)+i\beta=(a-i)\beta^{-1}+i$, which implies that $a(\beta+1)=-i(\beta-1)$. If $i\in NQR(q)$, then the pair of orthogonal one-factorizations generated by the starter $S_\beta$ and $-S_\beta$ does include a cycle of length four; otherwise don't.
	\end{itemize}
	It is not hard to check that Case (iii) and Case (v) generate different cycles of length four. On the other hand, Case (iii) generate one cycle of length four and Case (i) and (v) generate the other cycle of length four. Therefore, the pair of orthogonal one-factorizations generated by the starter $S_\beta$ and $-S_\beta$ does include exactly two cycles of length four.
\end{proof}

\begin{teo}\label{theorem:main_2}
Let $q\equiv$3 (mod 8) be an odd prime power such that $q\equiv-1$ (mod 12). Then there is a pair of orthogonal $(2,C_4)$ one-factorization of $K_{q+1}$.
\end{teo}

\begin{proof}
Let $$M=\{\beta\in NQR(q)\setminus\{-2,-1,2\}: \beta-1,\beta+1\in NQR(q),\beta^2+1\in QR(q)\}.$$ By Lemma \ref{lemma:beta,betamedios,orto}  and Lemma \ref{lemma:final_2}, we only need to show that $M$ is not the empty set: see proof of Lemma \ref{lemma:conjunto_no_vacio}.    	
\end{proof}

Hence, by Theorem \ref{theorem:main_1} and Theorem \ref{theorem:main_2}, we have the following:

\begin{coro}
Let $q\equiv$3 (mod 8) be an odd prime power such that $q\equiv-1$ (mod 12). Then there are a pair of orthogonal one-factorization of $K_{q+1}$, $\mathcal{F}$ and $\mathcal{G}$, such that $F\cup G$ include at most two cycles of length four, for every $F,G\in\mathcal{F}\cup\mathcal{G}$.
\end{coro}

{\bf Acknowledgment}

\

Research was partially supported by SNI and CONACyT.

\end{document}